\documentclass{amsart}

\usepackage{amsmath,amssymb,amsthm}
\usepackage{mathrsfs}
\usepackage{thmtools,thm-restate}

\usepackage[dvipsnames]{xcolor}
\usepackage{tikz-cd}
\usepackage{graphicx}
\usepackage{float}

\usepackage{booktabs}
\usepackage{multirow}
\usepackage{enumerate}

\usepackage{indentfirst}
\usepackage[numbers,sort]{natbib}
\usepackage{hyperref}

\declaretheorem[name=Theorem,numberwithin=section]{thm}
\declaretheorem[name=Proposition,sibling=thm]{prop}
\declaretheorem[name=Lemma,sibling=thm]{lem}
\declaretheorem[name=Corollary,sibling=thm]{cor}
\declaretheorem[style=definition,name=Problem,sibling=thm]{prob}
\declaretheorem[style=definition,name=Definition,sibling=thm]{defn}
\declaretheorem[style=remark,name=Remark,sibling=thm]{rem}

\numberwithin{equation}{section}

\newcommand{\innpr}[2]{\left\langle #1,#2\right\rangle}
\newcommand{\Ric}{\mathrm{Ric}}
\newcommand{\diff}{\mathrm{d}}

\begin{document}

\title[]{Principal Distribution Isomorphisms and Almost Hermitian geometry on Isoparametric Hypersurfaces}



\author[L.X. Xiao]{Lixin Xiao}
\address{School of Mathematical Sciences, Laboratory of Mathematics and Complex Systems, Beijing Normal University, Beijing 100875, P.R. CHINA.}
\email{lixinxiao@mail.bnu.edu.cn}

\author[W.J. Yan]{Wenjiao Yan}
\address{School of Mathematical Sciences, Laboratory of Mathematics and Complex Systems, Beijing Normal University, Beijing 100875, P.R. CHINA.}
\email{wjyan@bnu.edu.cn}

\author[W.J. Zhang]{Wenjin Zhang$^{*}$}
\address{School of Mathematical Sciences, Laboratory of Mathematics and Complex Systems, Beijing Normal University, Beijing 100875, P.R. CHINA.}
\email{wjinzhang@mail.bnu.edu.cn}

\thanks {$^{*}$ the corresponding author.}
\thanks{The project is partially supported by the NSFC (No. 12271038, 12526205).}


\date{}


\commby{}

\begin{abstract}
This paper investigates the isomorphisms between principal distributions $\mathcal{D}_k$ $(k=1,\dots 4)$ on OT--FKM type isoparametric hypersurfaces in spheres. We recover the isomorphism $\mathcal{D}_1 \cong \mathcal{D}_3$ established by Qian--Tang--Yan \cite{Q-T-Y 2}, and further construct the isomorphism  $\mathcal{D}_{2}\cong\mathcal{D}_{4}$ in specific cases. More significantly, we provide an explicit construction of a global vector bundle isomorphism $\mathcal{D}_1 \oplus \mathcal{D}_2 \cong \mathcal{D}_3 \oplus \mathcal{D}_4$ for all odd multiplicities $m$. As applications, we employ these isomorphisms to induce nearly K\"ahler structures on certain OT--FKM hypersurfaces. Finally, we prove that the $*$-Ricci curvature vanishes for any OT--FKM hypersurface admitting an almost Hermitian structure that interchanges principal distributions in pairs.
\end{abstract}

\maketitle



\section{Introduction}
The theory of isoparametric hypersurfaces in spheres has been a central topic in submanifold geometry and topology, initiating with the foundational work of E. Cartan. A hypersurface $M^n \subset S^{n+1}(1)$ is called isoparametric if it has constant principal curvatures. In his celebrated work \cite{Munz 1,Munz 2}, M\"unzner established a profound structural theorem, proving that the number $g$ of distinct principal curvatures is restricted to $1, 2, 3, 4$, or $6$. Moreover, the principal curvatures are given by $\lambda_{k}=\cot(\theta_{1}+\tfrac{k-1}{g}\pi)$ with $\theta_{1}\in (0,\tfrac{\pi}{g})$ for $1\leqslant k\leqslant g$, and their multiplicities satisfy $m_{k}=m_{k+2}$ (indices mod $g$). Consequently, the geometric data of such hypersurfaces is closely related with the multiplicity pair $(m_1, m_2)$. 

With M\"unzner's structural foundation established, the classification of isoparametric hypersurfaces became central challenge in this field. The initial phase of this program focused on restricting the admissible values of $(m_1, m_2)$ through topological constraints. Significant contributions regarding the topological invariants, cohomology rings, and homotopy properties were made by Abresch \cite{Abr}, Tang \cite{Tang}, Fang \cite{Fang, Fang99}. Based on these developments, a pivotal breakthrough for the $g=4$ case was made by Stolz \cite{Stolz}, who utilized these works to prove that the multiplicity pairs $(m_1, m_2)$ must coincide with those of the homogeneous examples or the inhomogeneous OT--FKM type constructed by Ferus, Karcher, and M\"unzner \cite{F-K-M}. Building on decades of cumulative efforts, the classification was finally completed in 2020: all isoparametric hypersurfaces are either homogeneous or of OT--FKM type (see \cite{C-C-J, Imm, CQS-1, CQS-2, CQS-3}). Notably, the OT--FKM type represents the vast majority, providing a rich source of manifolds with intricate geometric and topological properties.

A fundamental geometric object associated with the isoparametric theory is the \emph{principal distribution} $\mathcal{D}_k$, defined as the eigen-distribution corresponding to the principal curvature $\lambda_k$. These distributions are integrable, with leaves being totally geodesic spheres (see, for example, \cite{Chi}). Despite the classification of the hypersurfaces, the properties of these distributions as vector bundles remains a subject of great interest. A natural yet challenging problem is to classify these distributions as vector bundles over the isoparametric hypersurface and determine their isomorphism classes.

Regarding this problem, for $g=4$ and $g=6$, Abresch \cite{Abr} calculated the cohomology of the induced hypersurface $\overline{M}$ in the real projective space (as $M$ is invariant under the antipodal map) and the Stiefel--Whitney classes of the induced principal distributions $\overline{\mathcal{D}}_k$. He showed that for OT--FKM type hypersurfaces, the total Stiefel--Whitney classes satisfy $w(\overline{\mathcal{D}}_{1})=w(\overline{\mathcal{D}}_{3})=1$ and $w(\overline{\mathcal{D}}_{2})=w(\overline{\mathcal{D}}_{4})\neq 1$. 
Furthermore, in the $g=4$ case, there exists an involution $T:\overline{M}\to\overline{M}$, which maps $[x]$ to the equivalence class of the normal vector $\xi(x)$.
Abresch proved that $T^*$ acts as the identity on the $\mathbb{Z}_2$-cohomology. Fang \cite{Fang} further analyzed the Pontryagin classes, proving their triviality under the assumption that $T^*$ is the identity on the integral cohomology. See also \cite{W, Q-T-Y 1, G-T-Y} for other relevant aspects.

Motivated by Abresch's topological insights and the symmetry of the focal submanifolds, it is natural to conjecture that for OT--FKM type hypersurfaces, the principal distributions are pairwise isomorphic, i.e., $\mathcal{D}_{1}\cong\mathcal{D}_{3}$ and $\mathcal{D}_{2}\cong\mathcal{D}_{4}$. Recently, Qian, Tang and Yan \cite{Q-T-Y 2} established this isomorphism for $\mathcal{D}_{1}$ and  $\mathcal{D}_{3}$. Moreover, for the specific case $m_1=1$, they successfully constructed integrable almost complex structures by defining a correspondence mapping $\mathcal{D}_{1}$ to $\mathcal{D}_{3}$ and $\mathcal{D}_{2}$ to $\mathcal{D}_{4}$. For general multiplicities, however, the isomorphism between the remaining pair remains an open issue. They posed the following problem:
\begin{prob}(\cite{Q-T-Y 2})\label{prob1.1}
	For $m \geq 2$, are $\mathcal{D}_{2}$ and $\mathcal{D}_{4}$ isomorphic as vector bundles over $M$? \end{prob}

Beyond its intrinsic topological significance, an affirmative answer to Problem \ref{prob1.1} provides a natural link to differential geometry, particularly within the framework of almost Hermitian geometry. The isomorphisms between these principal distributions induce almost complex structures on $M$ in a natural way; by construction, such structures are inherently encoded with the geometric features of the hypersurface.

This framework allows for a systematic exploration of distinguished geometric structures, such as the nearly Kähler condition and the behavior of the $\ast$-Ricci curvature. In particular, by employing congruent pairs of OT--FKM type hypersurfaces, we construct explicit examples to demonstrate these properties:

\begin{restatable}{prop}{NearlyKahler}
\label{eg-nearlykahler}
	Isoparametric hypersurfaces of OT--FKM type with multiplicities $(m_1,m_{2})\in\{(1,2), (1,6), (2,5),(3,4)\}$ admit a nearly K\"ahler structure.
\end{restatable}

\begin{restatable}{prop}{StarRicVanish}
\label{*ric vanish}
	Let $M$ be an isoparametric hypersurface of OT--FKM type  in the sphere. If $M$ admits an almost Hermitian structure mapping $\mathcal{D}_{1}$ to $\mathcal{D}_{3}$ and $\mathcal{D}_{2}$ to $\mathcal{D}_{4}$ with respect to the induced metric, then the corresponding $*$-Ricci curvature vanishes identically.
\end{restatable}
In general, the existence of an isomorphism between $\mathcal{D}_{2}$ and $\mathcal{D}_{4}$ remains an open question. A natural relaxation of Problem \ref{prob1.1} is to ask:
\begin{prob}\label{prob1.2}
Is the direct sum $\mathcal{D}_{1}\oplus\mathcal{D}_{2}$ isomorphic to $\mathcal{D}_{3}\oplus\mathcal{D}_{4}$?
\end{prob}

In the present paper, we provide a partial affirmative answer to Problem \ref{prob1.2} for isoparametric hypersurfaces of OT--FKM type. The main result is as follows:
\begin{restatable}{thm}{MainThm}
\label{main thm}
	For isoparametric hypersurfaces of OT--FKM type, if $m$ is odd, there exists a global vector bundle isomorphism from $\mathcal{D}_{1}\oplus \mathcal{D}_{2}$ to $\mathcal{D}_{3}\oplus \mathcal{D}_{4}$.
\end{restatable}

The paper is organized as follows. In Section 2, we analyze the structural properties of principal distributions on OT--FKM type isoparametric hypersurfaces and present the proof of Theorem \ref{main thm}. Section 3 is devoted to applications of the isomorphisms between principal distributions within the framework of almost Hermitian geometry, where we investigate specific almost complex structures, the existence of nearly K\"ahler structures, and the bahavior of the $*$-Ricci curvatures.

\section{Principal Distributions}

We begin by briefly reviewing the construction of isoparametric hypersurfaces of OT--FKM type, as introduced in \cite{F-K-M}. 

A \textit{symmetric Clifford system} on $\mathbb{R}^{2l}$ is defined as an $(m+1)$-tuple $\{P_{0},\dots,P_{m}\}$ of symmetric orthogonal matrices satisfying $P_{i}P_{j}+P_{j}P_{i}=2\delta_{ij}I$. The dimension $2l$ is constrained by the representation theory of Clifford algebras; specifically, $l = k\delta(m)$ for some positive integer $k$, where $\delta(m)$ denotes the dimension of the irreducible module of the Clifford algebra $\mathcal{C}_{m-1}$. The values of $\delta(m)$ satisfy the periodicity condition $\delta(m+8) = 16\delta(m)$, and are given as follows:

\begin{table}[h]
	\centering
	\begin{tabular}{c ccccccccc} 
		\toprule
		$m$ & 1 & 2 & 3 & 4 & 5 & 6 & 7 & 8 & $\cdots m+8$\\
		\midrule
		$\delta(m)$ & 1 & 2 & 4 & 4 & 8 & 8 & 8 & 8& $\cdots 16\delta(m)$ \\
		\bottomrule
	\end{tabular}
\end{table}

Associated with a symmetric Clifford system is the Cartan--M\"unzner polynomial $F: \mathbb{R}^{2l} \to \mathbb{R}$, which is a homogeneous polynomial of degeree $4$ defined in \cite{F-K-M} as:
\begin{equation}
	F(x) = |x|^{4} - 2\sum_{i=0}^{m} \langle P_{i}x, x \rangle^{2}.
\end{equation}
The restriction of $F$ to the unit sphere $S^{2l-1}$ yields a function $f = F|_{S^{2l-1}}$ with image $\mathrm{Im}(f) = [-1,1]$. An isoparametric hypersurface $M^{2l-2}$ is obtained as a regular level set of $f$, specifically $M = f^{-1}(\cos 4\theta)$ for a constant $\theta \in (0, \tfrac{\pi}{4})$. The hypersurface $M$ possesses four distinct principal curvatures given by $\lambda_{k} = \cot(\theta + \tfrac{k-1}{4}\pi)$ for $k=1,\dots,4$. The associated multiplicities are $(m_1, m_2, m_1, m_2)$, where $m_{1}=m$ and $m_{2}=l-m-1$.

The function $f$ admits two singular level sets, $M_{\pm}=f^{-1}(\pm1)$, referred to as the focal submanifolds. These are minimal submanifolds of the unit sphere. The geometry of the isoparametric family is characterized by the fibrations $M \to M_{\pm}$, where the fibers are $S^{m_1}$ and $S^{m_2}$, respectively. Crucially, the normal bundle of $M_{+}$ in $S^{2l-1}$ is trivial and admits a global orthonormal frame $\{P_{0}x, \dots, P_{m}x\}$ at any $x\in M_+$.

\subsection{Isomorphisms of Principal Distributions}\
Let $M=f^{-1}(\cos 4\theta)$ be an isoparametric hypersurface, and let $\xi(x)=\frac{\nabla^{S} f}{|\nabla^{S} f|}$ denote the unit normal vector at $x \in M$, where $\nabla^{S}$ is the spherical gradient. Explicitly, 
\begin{equation}
	\xi(x) = \frac{1}{\cos 2\theta}\left( x\sin 2\theta - \frac{1}{\sin 2\theta}\sum_{i=0}^{m}\langle P_{i}x, x \rangle P_{i}x \right).
\end{equation}
Let $\{e_{1}, \dots, e_{2l-2}\}$ be a local orthonormal frame of principal directions, hereafter referred to as a \textit{local principal frame}. Then we have
\begin{align*}
	\mathrm{d} x = \omega_{i}e_{i}, \quad \mathrm{d}\xi = -\lambda_{i}\omega_{i}e_{i},
\end{align*}
where $\omega_{i}$ is the 1-form dual to $e_{i}$ and $\lambda_{i}$ is the corresponding principal curvature.

For $k \in \{1,2,3,4\}$, let $t_{k} = \theta + \frac{k-1}{4}\pi$. We consider the parallel transportation  map $\varphi_{k}: M \to S^{2l-1}$ at distance $t_k$ along the normal direction $\xi(x)$ and the associated unit normal vector field $\xi_{k}: M \to S^{2l-1}$ defined by
\begin{equation}\label{def of phi_k xi_k}
	\begin{aligned}
		\varphi_{k}(x) &:= x\cos t_{k} + \xi(x)\sin t_{k},\\
		\xi_{k}(x) &:= -x\sin t_{k} + \xi(x)\cos t_{k}.
	\end{aligned}
\end{equation}
It is well-known that $\varphi_{1}(M) = \varphi_{3}(M) = M_{+}$ and $\varphi_{2}(M) = \varphi_{4}(M) = M_{-}$.

We introduce an operator $P: \mathbb{R}^{2l} \to \mathbb{R}^{2l}$ defined by
\begin{equation}\label{def of P} 
	P = \frac{1}{\sin 2\theta}\sum_{i=0}^{m}\langle P_{i}x, x \rangle P_{i}.
\end{equation}
Note that $P^{2}=I$, and $P$ lies on the unit sphere of the linear space $\mathrm{Span}\{P_{0},\dots,P_{m}\}$. We refer to this sphere as the \emph{Clifford sphere}, denoted by $\Sigma = \Sigma(P_{0}, \dots, P_{m})$. A straightforward computation yields the following identities:
\begin{equation}\label{phi_{k}} 
	\begin{aligned} 
		\varphi_{1}(x) &= \frac{1}{\cos 2\theta}(x\cos \theta - Px\sin \theta), \\
		\xi_{1}(x) &= \frac{1}{\cos 2\theta}(x\sin \theta - Px\cos \theta) = -P\varphi_{1}(x), \\
		\varphi_{2}(x) &= \frac{I-P}{\sqrt{2}}\varphi_{1}(x), \quad \xi_{2}(x) = -\frac{I+P}{\sqrt{2}}\varphi_{1}(x),\\
		\varphi_{3}(x) &= -P\varphi_{1}(x), \quad \quad \, \xi_{3}(x) = -\varphi_{1}(x),\\
		\varphi_{4}(x) &= -\frac{I+P}{\sqrt{2}}\varphi_{1}(x), \quad \xi_{4}(x) = -\frac{I-P}{\sqrt{2}}\varphi_{1}(x).
	\end{aligned}
\end{equation}	

\begin{lem}\label{four principal distributions}
	The four principal distributions satisfy the following decomposition relations regarding the normal spaces of the focal submanifolds:
	\begin{align*}
		N_{\varphi_{1}(x)}M_{+} &= \mathcal{D}_{1}|_{x} \oplus \mathrm{Span}\{\xi_{1}(x)\},\\
		N_{\varphi_{2}(x)}M_{-} &= \mathcal{D}_{2}|_{x} \oplus \mathrm{Span}\{\xi_{2}(x)\},\\
		N_{\varphi_{3}(x)}M_{+} &= \mathcal{D}_{3}|_{x} \oplus \mathrm{Span}\{\xi_{3}(x)\},\\
		N_{\varphi_{4}(x)}M_{-} &= \mathcal{D}_{4}|_{x} \oplus \mathrm{Span}\{\xi_{4}(x)\},
	\end{align*}
	where $N_{p}M_{\pm}$ denotes the normal space of $M_{\pm}$ in $S^{2l-1}$ at $p$.
\end{lem}

\begin{proof}
	Differentiation of $\varphi_{1}$ yields
	\begin{align*}
		\mathrm{d}\varphi_{1} = (\cos\theta - \lambda_{i}\sin\theta)\omega_{i}e_{i} = \frac{\sin(\theta_{i}-\theta)}{\sin\theta_{i}}\omega_{i}e_{i},
	\end{align*}
	where $\theta_{i}$ is determined by $\lambda_{i} = \cot\theta_{i}$. Since the principal curvatures of $M=f^{-1}(\cos4\theta)$ are given by $\lambda_{k} = \cot(\theta + \frac{k-1}{4}\pi)$, the coefficient $\sin(\theta_i - \theta)$ vanishes precisely for vectors in $\mathcal{D}_1$. Consequently, the tangent space of $M_+=\varphi_{1}(M)$ at $\varphi_1(x)$ is spanned by the remaining distributions:
	$$
	T_{\varphi_{1}(x)}M_{+} = (\mathcal{D}_{2} \oplus \mathcal{D}_{3} \oplus \mathcal{D}_{4})|_{x}.
	$$
	Decomposing the ambient space $\mathbb{R}^{2l}$ appropriately, we obtain
	$$
	\mathrm{Span}\{ \varphi_{1}(x)\} \oplus N_{\varphi_{1}(x)}M_{+} = \mathcal{D}_{1}|_{x} \oplus \mathrm{Span}\{ x, \xi(x)\}
	= \mathcal{D}_{1}|_{x} \oplus \mathrm{Span}\{ \varphi_{1}(x), \xi_{1}(x)\}.
	$$
	This implies
	\begin{align*} 
		N_{\varphi_{1}(x)}M_{+} = \mathcal{D}_{1}|_{x} \oplus \mathrm{Span}\{ \xi_{1}(x)\}.
	\end{align*}
	The decompositions for the other three principal distributions follow analogously.
\end{proof}

As a natural corollary of Lemma \ref{four principal distributions}, we obtain the following result, which was previously established by Qian, Tang, and Yan \cite{Q-T-Y 2}. The proof presented here provides an alternative approach utilizing the geometry of the Clifford sphere: 
\begin{cor}\label{D_{1} and D_{3}}
	For any isoparametric hypersurface of OT--FKM type, the distributions $\mathcal{D}_{1}$ and $\mathcal{D}_{3}$ are isomorphic.
\end{cor}

\begin{proof}
	From Lemma \ref{four principal distributions} and Equation \eqref{phi_{k}}, we have
	$$
	N_{\varphi_{1}(x)}M_{+} = \mathcal{D}_{1}|_{x} \oplus \mathrm{Span}\{ \xi_{1}(x)\} = \mathcal{D}_{1}|_{x} \oplus \mathrm{Span}\{ P\varphi_{1}(x)\}.
	$$
	The normal space of $M_+$ at $\varphi_{1}(x)$ in $S^{2l-1}$ is spanned by $\{P_0 \varphi_{1}(x), \dots, P_m \varphi_{1}(x)\}$. We choose an orthonormal basis $\{R_0, \dots, R_m\}$ for the space $\mathrm{Span}\{P_0, \dots, P_m\}$ such that $R_0 = P$. Specifically, we may write 
	\begin{equation}\label{R_{0}=P}
		(R_0, \dots, R_m) = (P_0, \dots, P_m) A(x) \quad \mathrm{with} \ R_0 = P 
	\end{equation}
	for some $A(x) \in SO(m+1)$. Thus
	$$
	\mathrm{Span}\{ R_0 \varphi_{1}(x),\dots,R_m \varphi_{1}(x)\}= N_{\varphi_{1}(x)} M_{+} = \mathcal{D}_{1}|_x \oplus \mathrm{Span}\{ R_0 \varphi_{1}(x) \},
	$$
	which implies
	\begin{equation}\label{D_1}
		\mathcal{D}_{1}|_x = \mathrm{Span}\{ R_1 \varphi_{1}(x), \dots, R_m \varphi_{1}(x) \}.
	\end{equation}
	Similarly, one finds
	\begin{equation}\label{D_3}
		\mathcal{D}_{3}|_x = \mathrm{Span}\{ R_1 \varphi_{3}(x), \dots, R_m \varphi_{3}(x) \}.
	\end{equation}
	Since $\{R_a\}_{a=1}^m$ are orthogonal to $R_0 = P$ in the Clifford system, they anti-commute with $P$. Using \eqref{phi_{k}}, we observe $R_a \varphi_{3} = -R_a R_0 \varphi_{1} = R_0 R_a \varphi_{1}$. Therefore, the map $R_0 = P$ induces a global vector bundle isomorphism between $\mathcal{D}_{1}$ and $\mathcal{D}_{3}$.

\end{proof}

According to \cite{F-K-M}, there are exactly four pairs of congruent isoparametric families with multiplicity pairs $(m_1, m_2)$:
$$
(1,2) \leftrightarrow (2,1), \quad (1,6) \leftrightarrow (6,1), \quad (2,5) \leftrightarrow (5,2), \quad (3,4) \leftrightarrow (4,3)^{\mathrm{ind}}.
$$
Here, $(4,3)^{\mathrm{ind}}$ denotes the indefinite case, which is not congruent to the definite case $(4,3)^{\mathrm{def}}$, despite sharing the same multiplicities. Duality implies that $M_{-}$ in a case $(m_1, m_2)$ is congruent to $M'_{+}$ in the dual case $(m_2, m_1)$. Consequently, if $M'_+$ has a trivial normal bundle, so does $M_-$. This leads to the following result:

\begin{cor}\label{D_{2}and D_{4}}
	For the cases $(m_{1},m_{2}) \in \{(1,2), (1,6), (2,5), (3,4)\}$, the principal distributions $\mathcal{D}_{2}$ and $\mathcal{D}_{4}$ are isomorphic.
\end{cor}

\begin{proof}
	We first address the cases $(m_{1},m_{2}) \in \{(1,6), (2,5), (3,4)\}$, which are all hypersurfaces in $S^{15}$. There exists a symmetric Clifford system $\{P_{0},\dots,P_{8}\}$ on $\mathbb{R}^{16}$ satisfying $\sum_{i=0}^{8}\langle P_{i}x, x \rangle^{2} = |x|^4$ (cf. \cite{F-K-M}). We define two isoparametric functions on $S^{15}$:
	\begin{align*}
		f &= 1 - 2\sum_{i=0}^{m}\langle P_{i}x, x \rangle^{2},\\
		f' &= 1 - 2\sum_{i=m+1}^{8}\langle P_{i}x, x \rangle^{2} = -f.
	\end{align*}
	For an isoparametric hypersurface $M = f^{-1}(\cos4\theta)$, there corresponds a hypersurface $M' = (f')^{-1}(\cos4(\frac{\pi}{4}-\theta))$ such that $M$ and $M'$ coincide as sets. The principal curvatures are related by
	$$
	\lambda_{k} = \cot\left(\theta+\frac{k-1}{4}\pi\right) = -\cot\left(\frac{\pi}{4}-\theta+\frac{4-k}{4}\pi\right) = -\lambda^{'}_{5-k},
	$$
	Here, the minus sign arises from the opposite orientation of the normal vectors. Thus, the principal distributions of $M$ and $M'$ correspond according to $\mathcal{D}_{2} = \mathcal{D}_{3}^{'}$ and $\mathcal{D}_{4} = \mathcal{D}_{1}^{'}$. By an argument analogous to the proof of Corollary \ref{D_{1} and D_{3}}, the isomorphism between $\mathcal{D}_{2}$ and $\mathcal{D}_{4}$ is induced by the operator
	$$
	Q = P' = \frac{1}{\cos2\theta}\sum_{i=m+1}^{8}\langle P_{i}x, x \rangle P_{i}.
	$$
	For the remaining case $(m_{1},m_{2})=(1,2)$, the proof follows identically from the existence of a symmetric Clifford system $\{P_{0},\dots,P_{4}\}$ on $\mathbb{R}^{8}$.

\end{proof}

\subsection{Isomorphisms between $\mathcal{D}_{1}\oplus\mathcal{D}_{2}$ and $\mathcal{D}_{3}\oplus\mathcal{D}_{4}$}

Although the existence of an isomorphism between $\mathcal{D}_{2}$ and $\mathcal{D}_{4}$ in the general case remains an open problem, we establish a partial result in this subsection: the vector bundles $\mathcal{D}_{1}\oplus\mathcal{D}_{2}$ and $\mathcal{D}_{3}\oplus\mathcal{D}_{4}$ are globally isomorphic when $m$ is odd.

Recall Lemma \ref{four principal distributions} and the characterization of the fibers provided in \cite[p. 488]{F-K-M}. The fiber $\mathcal{D}_{2}|_{x}$ can be expressed as:
\begin{align*}
	\mathcal{D}_{2}|_{x} &= \left\{ v\in N_{\varphi_{2}(x)}M_{-} \mid v \perp \xi_{2}(x) \right\} \\
	&= \left\{ v\in E_{+}(P) \mid v \perp \frac{I+P}{\sqrt{2}}\Sigma_{P}\varphi_{1}(x) \quad \text{and} \quad v \perp \frac{I+P}{\sqrt{2}}\varphi_{1}(x) \right\},
\end{align*}
where $\Sigma_{P} = \{ Q \in \Sigma \mid \langle Q, P \rangle = 0 \}$ is the great hypersphere in the Clifford sphere $\Sigma$ orthogonal to $P$, and $\Sigma_{P}\varphi_{1}(x)=\{ Q\varphi_1(x) \mid Q \in \Sigma_P \}$ spans the distribution $\mathcal{D}_{1}$.

Consequently, the $\pm 1$ eigenspaces of $P$ admit the following orthogonal decompositions:
\begin{align*}
	E_{+}(P) &= \frac{I+P}{\sqrt2}\mathrm{Span}\{\varphi_1\} \oplus \frac{I+P}{\sqrt2}\mathcal{D}_{1} \oplus \mathcal{D}_{2},\\[4pt]
	E_{-}(P) &= \frac{I-P}{\sqrt2}\mathrm{Span}\{\varphi_1\} \oplus \frac{I-P}{\sqrt2}\mathcal{D}_{1} \oplus \mathcal{D}_{4}.
\end{align*}


\MainThm*

\begin{proof}
	By Corollary \ref{D_{1} and D_{3}}, we have $\mathcal{D}_{1}\cong\mathcal{D}_{3}$. Thus, it suffices to construct an isomorphism between $\mathcal{D}_{1}\oplus \mathcal{D}_{2}$ and $\mathcal{D}_{1}\oplus \mathcal{D}_{4}$.
	
	Let $x$ be a point in $M$. We choose a local section $R_1$ as in \eqref{R_{0}=P} such that $R_1 \in \Sigma_P$ (i.e., $R_1$ anti-commutes with $P=R_0$). The operator $R_1$ maps the eigenspace $E_+(P)$ isometrically onto $E_-(P)$. We define the maps $\rho_{\pm}$ as the specific linear isometries identifying the geometric distributions with the algebraic eigenspaces:
	$$
	\rho_{\pm}:= \left( \frac{I\pm P}{\sqrt2} \right) \oplus \left( \frac{I\pm P}{\sqrt2} \right) \oplus \operatorname{Id}_{\mathcal{D}_{2}\, (\text{or }\mathcal{D}_{4})}.
	$$
	Consider the following commutative diagram:
	\begin{equation}\label{diagram}
		\begin{tikzcd}[row sep=large, column sep=huge]
			{\mathrm{Span}\{\varphi_1\}\oplus\mathcal{D}_1\oplus\mathcal{D}_2} && {E_+(P)} \\
			{\mathrm{Span}\{\varphi_1\}\oplus\mathcal{D}_1\oplus\mathcal{D}_4} && {E_-(P)} 
			\arrow["{\rho_+}", from=1-1, to=1-3]
			\arrow["\sigma"', dashed, from=1-1, to=2-1]
			\arrow["{R_1}", from=1-3, to=2-3]
			\arrow["{\rho_-}"', from=2-1, to=2-3]
		\end{tikzcd}
	\end{equation}
	The composition $\sigma = \rho_-^{-1} \circ R_1 \circ \rho_+$ induces a linear isometry between the vector bundles. A direct calculation reveals the action of $\sigma$ on the first two summands:
	$$
	\sigma(\varphi_{1}) = R_{1}\varphi_{1} \quad \text{and} \quad \sigma(R_{1}\varphi_{1}) = \varphi_{1}.
	$$
	To isolate the isomorphism on the desired distributions, we decompose $\mathcal{D}_1$. Let $\mathcal{D}_{1}^{0}$ be the orthogonal complement of $\mathrm{Span}\{ R_{1}\varphi_{1}\}$ in $\mathcal{D}_{1}$. The restriction of $\sigma$ yields an isomorphism:
	\begin{equation}\label{sigma}
		\sigma|_{\mathcal{D}_{1}^{0}\oplus \mathcal{D}_{2}}: \mathcal{D}_{1}^{0}\oplus \mathcal{D}_{2} \longrightarrow \mathcal{D}_{1}^{0}\oplus \mathcal{D}_{4}.
	\end{equation}
	We now define a modified map $\tilde{\sigma}: \mathcal{D}_{1}\oplus \mathcal{D}_{2} \to \mathcal{D}_{1}\oplus \mathcal{D}_{4}$ by setting:
	$$
	\tilde{\sigma}(v) = 
	\begin{cases}
		\sigma(v) & \text{if } v \in \mathcal{D}_{1}^{0} \oplus \mathcal{D}_{2}, \\
		R_{1}\varphi_{1} & \text{if } v = R_{1}\varphi_{1}.
	\end{cases}
	$$
	This $\tilde{\sigma}$ provides a local isomorphism between $\mathcal{D}_{1}\oplus \mathcal{D}_{2}$ and $\mathcal{D}_{1}\oplus \mathcal{D}_{4}$.

	To extend this construction $\tilde{\sigma}$ globally, we assume $m$ is odd, in which case the Clifford sphere $\Sigma(P_{0}, \dots, P_{m})\cong S^{m}$ admits a nowhere-vanishing tangent vector field $V$. 
	
	For any $x \in M$, let $P(x)$ be the corresponding point on the Clifford sphere $\Sigma$. We define a global operator field $R_{1}(x) := V(P(x))$. Since $V$ is tangent to the sphere, $\langle R_1(x), P(x) \rangle = 0$, ensuring that $R_1(x) \in \Sigma_{P(x)}$ for all $x$. Using this global section $R_1$, the previous construction yields a globally defined isomorphism $\widetilde{\sigma}: \mathcal{D}_{1}\oplus \mathcal{D}_{2} \xrightarrow{\cong} \mathcal{D}_{1}\oplus \mathcal{D}_{4}$.
	
\end{proof}

\begin{rem}
	In the case $m=1$, the subspace $\mathcal{D}_{1}^{0}$ vanishes. Consequently, the map $\sigma$ constructed in \eqref{sigma} restricts directly to an isomorphism $\mathcal{D}_{2} \xrightarrow{\cong} \mathcal{D}_{4}$. Explicitly, if we parametrize the Clifford circle as $P = (\cos \alpha) P_0 + (\sin \alpha) P_1$, we may choose the tangent vector field $R_{1} = (\sin \alpha)P_{0} - (\cos \alpha)P_{1}$. Since $\mathcal{D}_{2} \subset E_{+}(P)$, the restriction is given by $\sigma|_{\mathcal{D}_{2}} = R_{1} |_{\mathcal{D}_{2}} = P_{0}P_{1}$, where we utilized the anti-commutativity of the Clifford matrices.
\end{rem}

\section{Applications to Almost Complex Geometry}

Let $M^{2n}$ be an even-dimensional smooth manifold. An endomorphism $J: TM\to TM$ satisfying $J^{2}=-\operatorname{Id}$ is termed an \textit{almost complex structure}. Such a structure is said to be \textit{integrable} if $M$ admits a complex atlas compatible with $J$. The celebrated Newlander--Nirenberg theorem states that an almost complex structure $J$ is integrable if and only if its Nijenhuis tensor $N$ vanishes identically, where $N$ is defined for smooth vector fields $X, Y \in \Gamma(TM)$ by
\begin{equation*}
	N(X,Y)=[JX,JY]-J[JX,Y]-J[X,JY]-[X,Y].
\end{equation*}

Tang and Yan \cite{T-Y} initiated the investigation of such structures on isoparametric hypersurfaces of OT--FKM type, providing explicit constructions of integrable almost complex structures in specific cases. Subsequently, Qian, Tang, and Yan \cite{Q-T-Y 2} established that every isoparametric hypersurface with $g=4$ admits an almost complex structure. A natural approach to constructing such structures is via bundle isomorphisms between pairs of principal distributions. For the case $m=1$, they demonstrated that a specific isomorphism mapping $\mathcal{D}_1$ to $\mathcal{D}_3$ and $\mathcal{D}_2$ to $\mathcal{D}_4$ yields an integrable structure after a suitable modification with constant coefficients. A compelling open question is whether this construction can be generalized to cases where $m>1$.

We retain the notation from Section 2. Let $\{e_{1},\dots,e_{2(m_{1}+m_{2})}\}$ be a local principal frame adapted to the distributions. We adopt the index convention $e_{\bar{i}} := e_{(m_{1}+m_{2})+i}$ and assume the frame is ordered such that:
$$
\operatorname{Span}\{e_{a}\}_{a=1}^{m_1} = \mathcal{D}_{1}, \ \operatorname{Span}\{e_{\alpha}\}_{\alpha=m_1+1}^{m_1+m_2} = \mathcal{D}_{2}, \ \operatorname{Span}\{e_{\bar{a}}\} _{a=1}^{m_1} = \mathcal{D}_{3}, \ \operatorname{Span}\{e_{\bar{\alpha}}\}_{\alpha=m_1+1}^{m_1+m_2} = \mathcal{D}_{4}.
$$
In light of Corollary \ref{D_{1} and D_{3}}, we further impose the relation $e_{\bar{a}}=-R_{0}e_{a}$ for $1\leqslant a\leqslant m_{1}$.

The integrable structure constructed in \cite{Q-T-Y 2} for $m=1$ relies on the mapping $e_{1} \mapsto \mu e_{\bar{1}}$ for an arbitrary nonzero constant $\mu$. The following proposition provides an evidence that a direct generalization of this mapping for $m=3$ fails to be integrable.

\begin{prop}
	Let $m=3$. Consider an almost complex structure $J$ that maps $\mathcal{D}_1$ to $\mathcal{D}_3$ and $\mathcal{D}_2$ to $\mathcal{D}_4$. If $J$ satisfies $Je_{a}=\mu e_{\bar{a}}$ for $a=1,2,3$ with a constant $\mu\neq0$, then $J$ is not integrable.
\end{prop}

\begin{proof}
	Let $\omega_{ij}=\sum_{k}\omega_{ijk}\omega_{k}$ denote the connection 1-forms, and let $h_{i}=\lambda_{i}\omega_{i}$ be the second fundamental forms. The Codazzi equation $\mathrm{d}h_{i}=\sum_{j}\omega_{ij}\wedge h_{j}$ implies the symmetry condition:
	$$
	\sum_{j}(\lambda_{i}-\lambda_{j})\omega_{j}\wedge\omega_{ij}=0, \quad \text{or} \quad (\lambda_{i}-\lambda_{j})\omega_{ijk}=(\lambda_{i}-\lambda_{k})\omega_{ikj}.
	$$
	Consequently, for distinct principal curvatures $\lambda_i \neq \lambda_k$, we have
	\begin{equation}\label{commute index}
		\omega_{ikj}=\frac{\lambda_{i}-\lambda_{j}}{\lambda_{i}-\lambda_{k}}\omega_{ijk}.
	\end{equation}
	In particular, if $\lambda_{i}=\lambda_{k}\neq\lambda_{j}$, it follows that $\omega_{ijk}=0$ (cf. \cite{C-C-J}).

	Recall from \eqref{R_{0}=P} that we may write the frame of the Clifford system as $(R_0, \dots, R_m) = (P_0, \dots, P_m) A(x)$ for some smooth $A: U \to SO(m+1)$ on an open set $U\subseteq M$. The differential of this frame is given by $\diff R_{j} = \sum_{i} R_{i} \tau_{ij}$, where $(\tau_{ij}) = A^{-1}\diff A$ is the Maurer--Cartan form. For any $1\leqslant a\leqslant m$, utilizing the explicit form of $R_0 = P$ defined in \eqref{def of P}, we compute:
	\[\tau_{a0} = \sum_{i} A_{ia}\diff A_{i0} = \frac{1}{\sin2\theta}\sum_{i}A_{ia}\,\diff \innpr{P_{i}x}{x} = \frac{2}{\sin2\theta}\sum_{i}A_{ia}\innpr{P_{i}x}{\mathrm{d}x}= \frac{2}{\sin2\theta}\innpr{R_{a}x}{\mathrm{d}x}.\]
	Expanding $\tau_{a0}$ via the expression $x = \varphi_{1}\cos\theta - \xi_{1}\sin\theta$ (with $\xi_{1}=-R_{0}\varphi_{1}$) yields:	
	\begin{equation}\label{tau}
		\tau_{a0} = \frac{2}{\sin2\theta}\langle R_{a}(\varphi_{1}\cos\theta + R_{0}\varphi_{1}\sin\theta), \mathrm{d}x \rangle = \frac{1}{\sin\theta}\omega^{a} + \frac{1}{\cos\theta}\omega^{\bar{a}}.
	\end{equation}

	We now evaluate the component $\langle N(e_{a},e_{b}),e_{c}\rangle$ of the Nijenhuis tensor for indices $a,b,c \in \{1, \dots, m_1\}$. Since $J e_a = \mu e_{\bar{a}}$ and $e_c \in \mathcal{D}_1$, the definition of $N$ simplifies. A lengthy but straightforward calculation using \eqref{commute index} yields:
	\begin{align*}
		\langle N(e_{a},e_{b}),e_{c}\rangle &= \mu^2 \langle [e_{\bar{a}},e_{\bar{b}}], e_c \rangle  - \mu \langle J[ e_{\bar{a}},e_{b}], e_c \rangle -  \mu\langle J[e_{a},  e_{\bar{b}}], e_c \rangle - \langle [e_{a},e_{b}], e_c \rangle\\
		&= \left( \omega_{b\bar{c}\bar{a}}-\omega_{\bar{a}\bar{c}b}\right) - \left( \omega_{a\bar{c}\bar{b}}-\omega_{\bar{b}\bar{c}a}\right) - \left(\omega_{bca}-\omega_{acb}\right) \\
		&= (\omega_{\bar{b}\bar{c}a}-\omega_{bca}) - (\omega_{\bar{a}\bar{c}b}-\omega_{acb}).
	\end{align*}
	Using the relation $\diff e_{\bar{a}}=-\diff (R_{0}e_{a}) = -\sum_c \tau_{c0}R_{c}e_{a} - R_{0}\,\diff e_{a}$, we express the connection forms with barred indices as:
	\begin{equation}\label{omega abar bbar}
		\omega_{\bar{a}\bar{b}} = \langle \diff e_{\bar{a}}, e_{\bar{b}} \rangle = \tau_{c0}\langle R_{c}R_{a}\varphi_{1}, R_{0}R_{b}\varphi_{1}\rangle + \omega_{ab}.
	\end{equation}
	Specifying the indices to $1, 2, 3$, we find:
	\begin{align*}
		\langle N(e_{1},e_{2}),e_{3}\rangle &= (\omega_{\bar{2}\bar{3}1}-\omega_{231}) - (\omega_{\bar{1}\bar{3}2}-\omega_{132})\\
		&= \tau_{10}(e_{1})\langle R_{1}R_{2}\varphi_1, R_{0}R_{3}\varphi_1\rangle - \tau_{20}(e_{2})\langle R_{2}R_{1}\varphi_1, R_{0}R_{3}\varphi_1\rangle \\
		&= \frac{1}{\sin\theta} \left( \langle R_1 R_2 \varphi_1, R_0 R_3 \varphi_1 \rangle + \langle R_2 R_1 \varphi_1, R_0 R_3 \varphi_1 \rangle \right) \\
		&= -\frac{2}{\sin\theta}\langle R_{0}R_{1}R_{2}R_{3}\varphi_1, \varphi_1\rangle\\
		&= -\frac{2}{\sin\theta\,\cos2\theta}\langle P_{0}P_{1}P_{2}P_{3}x, x\rangle.
	\end{align*}
	The function $x \mapsto \langle P_{0}P_{1}P_{2}P_{3}x, x\rangle$ is not identically zero; its image is $[-\cos2\theta,\cos2\theta]$. Thus, the Nijenhuis tensor does not vanish, and $J$ is not integrable.

\end{proof}

This example illustrates that for $m>1$, a direct isomorphism between the principal distributions does not automatically induce an integrable almost complex structure, highlighting the complexity of the $m>1$ case compared to $m=1$. Nevertheless, such structures possess certain significant geometric properties.

\subsection{Nearly K\"ahler Structures}
\label{subsec--nearly kahler}
For isoparametric hypersurfaces of OT--FKM type with multiplicity pairs $(m_{1},m_{2})=$ $(1,2)$, $(1,6)$, $(2,5)$, or $(3,4)$, the isomorphisms established in Corollaries \ref{D_{1} and D_{3}} and \ref{D_{2}and D_{4}} give rise to natural almost complex structures.

We explicitly illustrate this construction for the case $(m_1, m_2) = (3,4)$. Let the Clifford system be generated by $\{P_0, \dots, P_7\}$. We define two operators $R_0$ and $Q_0$ acting on $TM$ by the following linear combinations:
$$
R_{0} = \frac{1}{\sin 2\theta}\sum_{i=0}^{3}\langle P_{i}x, x \rangle P_{i}, \quad \text{and} \quad Q_{0} = \frac{1}{\cos 2\theta}\sum_{j=4}^{7}\langle P_{j}x, x \rangle P_{j}.
$$
These operators induce an almost complex structure $J$ defined by:
\begin{equation}\label{acs example}
	Jv = \begin{cases}
		-R_{0}v \in \mathcal{D}_{3}, & \text{if } v\in \mathcal{D}_{1}, \\
		-Q_{0}v \in \mathcal{D}_{4}, & \text{if } v\in \mathcal{D}_{2}, \\
		\phantom{-}R_{0}v \in \mathcal{D}_{1}, & \text{if } v\in \mathcal{D}_{3}, \\
		\phantom{-}Q_{0}v \in \mathcal{D}_{2}, & \text{if } v\in \mathcal{D}_{4}.
	\end{cases}
\end{equation}
One verifies directly that $J$ is compatible with the induced metric $g$, rendering $(M, J, g)$ an almost Hermitian manifold. Let $\Phi(X,Y) = g(JX, Y)$ denote the fundamental 2-form (or K\"ahler form). The manifold is termed \textit{almost K\"ahler} if $\diff\Phi = 0$, and \textit{nearly K\"ahler} if the covariant derivative $\nabla \Phi$ is totally skew-symmetric (cf. \cite{G-H}).

We now demonstrate that for certain almost Hermitian isoparametric hypersurfaces, the  K\"ahler form satisfies the nearly K\"ahler condition.

\NearlyKahler*

\begin{proof}
	We present the proof for the case $(m_{1},m_{2})=(3,4)$ with $J$ given by \eqref{acs example}; the arguments for the remaining cases follow analogously. To verify the nearly Kähler condition, we demonstrate the skew-symmetry of $\nabla\Phi$.
	
	Let $\{e_{1},\dots,e_{14}\}$ be a local principal frame adapted to the distributions such that $e_{a} \in \mathcal{D}_{1}$ for $1 \leqslant a \leqslant 3$, and $e_{\alpha} \in \mathcal{D}_{2}$ for $4 \leqslant \alpha \leqslant 7$. We adopt the notation $e_{\bar{k}} := J e_k$ for $1\leqslant k\leqslant 14$. Specifically, $e_{\bar{a}} =Je_a= -R_0 e_a \in \mathcal{D}_3$ and $e_{\bar{\alpha}} = -Q_0 e_\alpha \in \mathcal{D}_4$.
	Expressing the covariant derivative of the fundamental form in terms of connection coefficients $\omega_{ijk}$, we have:
	\begin{align*}
	G_{ijk} &:= (\nabla_{e_{i}}\Phi)(e_{j},e_{k}) \\
	&= -\innpr{J\nabla_{e_{i}}e_{j}}{e_{k}} - \innpr{Je_{j}}{\nabla_{e_{i}}e_{k}} \\
	&= \omega_{j\bar{k}i} - \omega_{k\bar{j}i}.
\end{align*}
The condition that $\nabla \Phi$ is totally skew-symmetric is equivalent to $(\nabla_X \Phi)(X, Y) = 0$ for all $X, Y$. Thus, it suffices to show that $G_{iij} = \omega_{i\bar{j}i} - \omega_{j\bar{i}i} = 0$ for all indices $i, j$.
	We analyze the vanishing of this term by distinguishing the relative positions of indices $i$ and $j$ with respect to the principal distributions. Recall that the Codazzi equation implies $(\lambda_A - \lambda_B)\omega_{ABC} = (\lambda_A - \lambda_C)\omega_{ACB}$ as shown in \eqref{commute index}. In particular, if $\lambda_A = \lambda_B \neq \lambda_C$, then $\omega_{ABC} = 0$.
	
	\noindent\textbf{Case 1: $e_i$ and $e_j$ belong to the same distribution.}
	Here $\lambda_i = \lambda_j$. In this case, $G_{iij}$ vanishes identically by \eqref{commute index}.
	
	
	\noindent\textbf{Case 2: $e_j$ and $e_{\bar{i}}$ belong to the same distribution}.
	We compute
	\begin{align*}
	G_{aa\bar{b}} &= -\omega_{aba} - \omega_{\bar{b}\bar{a}a} = (\omega_{\bar{a}\bar{b}} - \omega_{ab})(e_{a}), \\
	G_{\bar{a}\bar{a}b} &= \omega_{\bar{a}\bar{b}\bar{a}} + \omega_{ba\bar{a}} = (\omega_{\bar{a}\bar{b}} - \omega_{ab})(e_{\bar{a}}), \\
	G_{\alpha\alpha\bar{\beta}} &= -\omega_{\alpha\beta\alpha} - \omega_{\bar{\beta}\bar{\alpha}a} = (\omega_{\bar{\alpha}\bar{\beta}} - \omega_{\alpha\beta})(e_{\alpha}), \\
	G_{\bar{\alpha}\bar{\alpha}\beta} &= \omega_{\bar{\alpha}\bar{\beta}\bar{\alpha}} + \omega_{\beta\alpha\bar{\alpha}} = (\omega_{\bar{\alpha}\bar{\beta}} - \omega_{\alpha\beta})(e_{\bar{\alpha}}).
\end{align*}
Since $e_{\bar{a}}=-R_{0}e_{a}$, we derive from \eqref{omega abar bbar} that,
\[
G_{aa\bar{b}} = -\tau_{c0}(e_{a})\innpr{R_{c}e_{a}}{e_{\bar{b}}} = -\frac{1}{\sin\theta}\innpr{R_{a}R_{a}\varphi_{1}}{e_{\bar{b}}} = 0,
\]
since $R_a^2 = I$ and $\varphi_{1}$ is normal to $M$. Similarly, $G_{\bar{a}\bar{a}b} = 0$. The calculations for the remaining subcases follow analogously.
	
	\noindent\textbf{Case 3: $e_{i}, e_{\bar{i}}, e_{j}, e_{\bar{j}}$ belong to distinct principal distributions.}
We write $x \sim y$ to denote that quantities $x$ and $y$ differ by a non-zero scalar factor. From \eqref{commute index}, we have
$$G_{iij}=-\omega_{j\bar{i}i}\sim \omega_{i\bar{i}j}.$$
More specifically,
\begin{align*}
	&G_{aaj}\sim \omega_{a\bar{a}j}\sim G_{\bar{a}\bar{a}j},\\
	&G_{\alpha\alpha j}\sim\omega_{\alpha\bar{\alpha} j}\sim G_{\bar{\alpha}\bar{\alpha}j}.
\end{align*}
Following from \eqref{D_1} and \eqref{D_3}, we 
assume further
$$
e_{a}=R_{a}\varphi_{1} \quad \text{and} \quad e_{\bar{a}}=-R_{0}R_{a}\varphi_{1} \quad \text{for} \quad 1\leqslant a\leqslant 3.
$$
Recall that $\diff e_{a} = \diff (R_{a}\varphi_{1}) = \sum_c \tau_{ca}e_{c} + R_{a}\diff \varphi_{1}$. It follows that the connection 1-form satisfies
$$
\omega_{a\bar{b}} = \innpr{\diff e_{a}}{e_{\bar{b}}} = \innpr{\diff\varphi_{1}}{R_{a}R_{b}R_{0}\varphi_{1}}.
$$
Consequently, $\omega_{a\bar{a}}$ vanishes because $R_{0}\varphi_{1}$ is a normal vector to $M_{+}$ at $\varphi_{1}$. Similarly, we have $\omega_{\alpha\bar{\alpha}} = 0$. 	

Summarizing the above, we have $G_{iij} = 0$ for all possible choices of $e_{i}$ and $e_{j}$, which verifies the claim for $(m_1, m_2) = (3,4)$. As similar arguments apply to the remaining cases $(m_1, m_2) \in\{ (1,2), (1,6),(2,5)\}$, the proof is now complete.

\end{proof}

\begin{rem}
	Topological constraints prevent these hypersurfaces from being K\"ahler, as the odd Betti numbers of a compact K\"ahler manifold must be even. Moreover, according to the classification by Gray and Hervella \cite{G-H}, the structures constructed here are neither almost K\"ahler nor integrable.
\end{rem}

\subsection{The $*$-Ricci Curvature}

In the context of almost Hermitian geometry, it is natural to investigate curvature invariants that capture the interaction between the Riemannian metric and the almost complex structure. As observed in \cite{d-S}, the $*$\nobreakdash-Ricci curvature defined below is essentially a unique way to incorporate the almost complex structure into the Ricci curvature tensor.

\begin{defn}[see \cite{Y-K}]
	Let $(M^{2n}, J, g)$ be an almost Hermitian manifold. For any $p\in M$ and tangent vectors $X, Y\in T_{p}M$, the $*$-Ricci curvature is defined by
	\begin{equation}
		*\Ric(X, Y) := \operatorname{Tr}\left(Z \mapsto -\frac{1}{2}J \circ R(X, JY)Z\right).
	\end{equation}
\end{defn}

Unlike the standard Ricci curvature, the $*$-Ricci tensor is generally neither symmetric nor skew-symmetric. Let $\{e_1, \ldots, e_{2n}\}$ be an orthonormal basis of $T_pM$ such that $Je_i = e_{i+n}$ and $Je_{i+n} = -e_i$ for $1 \leqslant i \leqslant n$. The components of the $*$-Ricci tensor can be expressed as:
\begin{align*}
	*\Ric(X, Y) &= -\frac{1}{2}\sum_{i=1}^{n}\left(\innpr{JR(X, JY)e_{i}}{e_{i}} + \innpr{JR(X, JY)e_{i+n}}{e_{i+n}}\right)\\
	&= \sum_{i=1}^{n}R(X, JY, e_{i}, e_{i+n}).
\end{align*}

When $M$ is isometrically immersed into a Riemannian manifold $N$, the $*$-Ricci curvature satisfies the following structural equation.

\begin{lem}
	Let $M^{2n}$ be a submanifold of a Riemannian manifold $N^{2n+k}$, and let $\{\xi_1, \ldots, \xi_k\}$ be a local orthonormal frame of the normal bundle $T^\perp M$. Then
	\begin{equation}
		*\Ric(X, Y) = \sum_{i=1}^{n}\overline{R}(X, JY, e_{i}, e_{i+n}) - \sum_{\alpha=1}^{k} \innpr{JA_{\xi_{\alpha}}JA_{\xi_{\alpha}}(X)}{Y},
	\end{equation}
	where $\overline{R}$ denotes the Riemann curvature tensor of the ambient space $N$, and $A_{\xi_{\alpha}}$ is the shape operator associated with $\xi_{\alpha}$. In particular, if $M^{2n}$ is a hypersurface in the standard sphere $S^{2n+1}$ with unit normal $\xi$, then
	\begin{equation}\label{star ricci}
		*\Ric(X, Y) = \innpr{X}{Y} - \innpr{JA_{\xi}JA_{\xi}(X)}{Y}.
	\end{equation}
\end{lem}

\begin{proof}
	Let $B$ denote the second fundamental form of $M$. We evaluate the contribution of the extrinsic geometry using the Gauss equation:
	\begin{align*}
		\innpr{B(X, e_{i})}{B(JY,e_{i+n})} &= \sum_{\alpha}\innpr{B(X, e_{i})}{\xi_{\alpha}}\innpr{B(JY, e_{i+n})}{\xi_{\alpha}}\\
		&= \sum_{\alpha}\innpr{A_{\xi_{\alpha}}(X)}{e_{i}}\innpr{A_{\xi_{\alpha}}(JY)}{e_{i+n}}\\
		&= -\sum_{\alpha}\innpr{A_{\xi_{\alpha}}(X)}{e_{i}}\innpr{JA_{\xi_{\alpha}}(JY)}{e_{i}}.
	\end{align*}
	Similarly,
	$$
	\innpr{B(X, e_{i+n})}{B(JY,e_{i})} = \sum_{\alpha}\innpr{A_{\xi_{\alpha}}(X)}{e_{i+n}}\innpr{JA_{\xi_{\alpha}}(JY)}{e_{i+n}}.
	$$
	Substituting these expressions into the Gauss equation yields:
	\begin{align*}
		*\Ric(X, Y) &= \sum_{i=1}^{n} R(X, JY, e_{i}, e_{i+n})\\
		&= \sum_{i=1}^{n} \overline{R}(X, JY, e_{i}, e_{i+n}) + \sum_{i=1}^{n} \left( \innpr{B(X,e_{i})}{B(JY,e_{i+n})} - \innpr{B(X, e_{i+n})}{B(JY, e_{i})} \right)\\
		&= \sum_{i=1}^{n} \overline{R}(X, JY, e_{i}, e_{i+n}) - \sum_{\alpha=1}^{k} \sum_{l=1}^{2n} \innpr{A_{\xi_{\alpha}}(X)}{e_{l}}\innpr{JA_{\xi_{\alpha}}(JY)}{e_{l}}\\
		&= \sum_{i=1}^{n} \overline{R}(X, JY, e_{i}, e_{i+n}) - \sum_{\alpha=1}^{k} \innpr{A_{\xi_{\alpha}}(X)}{JA_{\xi_{\alpha}}(JY)}.
	\end{align*}
	Finally, utilizing the skew-symmetry of $J$ and the symmetry of $A_{\xi_\alpha}$, we observe that
	$$
	\innpr{A_{\xi_{\alpha}}(X)}{JA_{\xi_{\alpha}}(JY)} = -\innpr{JA_{\xi_{\alpha}}(X)}{A_{\xi_{\alpha}}(JY)} =- \innpr{A_{\xi_{\alpha}}JA_{\xi_{\alpha}}(X)}{JY} = \innpr{JA_{\xi_{\alpha}}JA_{\xi_{\alpha}}(X)}{Y}.
	$$
	The specific formula \eqref{star ricci} for $S^{2n+1}$ follows from the fact that $\sum_{i=1}^n \overline{R}(X, JY, e_i, e_{i+n}) = \innpr{X}{Y}$ when $N$ has constant sectional curvature $1$.

\end{proof}

For isoparametric hypersurfaces of OT--FKM type, the distinct principal curvatures satisfy the relations $\lambda_{1}\lambda_{3}=\lambda_{2}\lambda_{4}=-1$. This property leads to a vanishing result for the $*$-Ricci curvature.

\StarRicVanish*

\begin{proof}
	Any tangent vector $v$ admits a decomposition $v=\sum_{k=1}^{4}v_{k}$ with $v_{k}\in\mathcal{D}_{k}$. Consider an almost complex structure $J$ satisfying $J(\mathcal{D}_k)\subseteq\mathcal{D}_{k+2}$ (indices modulo 4). Consequently,
	$$
	JA_{\xi}JA_{\xi}(v) = \sum_{k=1}^{4}JA_{\xi}J(\lambda_{k}v_{k}) = \sum_{k=1}^{4}\lambda_{k+2}\lambda_{k}J^{2}(v_{k}).
	$$
	Since $\lambda_k \lambda_{k+2} = -1$ and $J^2 = -\operatorname{Id}$, we conclude that
	$$
	JA_{\xi}JA_{\xi}(v) = \sum_{k=1}^{4}(-1)(-v_k) = v.
	$$
	Substituting this into \eqref{star ricci}, we obtain $*\Ric(X, Y) = \innpr{X}{Y} - \innpr{X}{Y} = 0$.
\end{proof}

\begin{rem}
	The $*$-scalar curvature is defined as the trace of $*\Ric$. As a generalization of the classical Yamabe problem, the existence of a conformal metric with constant $*$-scalar curvature has been established by del Rio and Simanca \cite{d-S}. Since the scalar curvature and the $*$-scalar curvature are not necessarily simultaneously constant, Ge and Zhou \cite{G-Z} further investigated the difference between these two curvatures and derived a Yamabe-type theorem.  Proposition \ref{*ric vanish} provides explicit examples of hypersurfaces where both curvatures are constant (specifically, the $*$-scalar curvature vanishes) with respect to the induced metric.
\end{rem}

An almost Hermitian manifold $(M, J, g)$ is called \textit{weakly $*$-Einstein} if $*\Ric = \rho g$ for some smooth function $\rho$, and \textit{$*$-Einstein} if $\rho$ is constant. Evidently, the hypersurfaces satisfying Proposition \ref{*ric vanish} are $*$-Einstein, including concrete examples constructed in Subsection \ref{subsec--nearly kahler}.

We conclude by giving necessary conditions for an almost Hermitian hypersurface to be $*$-Einstein. Let $(M^{2n},J)$ be a hypersurface of $S^{2n+1}$ with unit normal vector field $\xi$. At any point $p \in M$, let $\lambda_{1} , \dots , \lambda_{k}$ be the distinct principal curvatures, and let $E_{\lambda_{1}}, \ldots, E_{\lambda_{k}}$ denote the corresponding eigenspaces.

\begin{prop} 
	Let $M^{2n} \subset S^{2n+1}$ be a hypersurface endowed with an  almost Hermitian structure $(J, g)$, where $g$ denotes the induced metric.
	\begin{enumerate}
		\item The tensor $*\Ric$ is symmetric if and only if $JA_{\xi}J(E_{\lambda}) \subseteq E_{\lambda}$ for every principal curvature $\lambda$.
		\item If $M$ is weakly $*$-Einstein, then at each point $p$, either $0$ is a principal curvature, or $J(E_{\lambda_i}) = E_{-c/\lambda_i}$ for some scalar $c \neq 0$ depending only on $p$.
		\item If $M$ is weakly $*$-Einstein, then $M$ is $*$-Einstein if and only if it has constant Gauss--Kronecker curvature. 
	\end{enumerate}
\end{prop}

\begin{proof}
	(1) Using \eqref{star ricci}, the skew-symmetric part of the $*$-Ricci tensor is
	\begin{align*}
		*\Ric(X, Y) - *\Ric(Y, X) &= \innpr{JA_{\xi}JA_{\xi}(Y)}{X} - \innpr{JA_{\xi}JA_{\xi}(X)}{Y}\\
		&= \innpr{A_{\xi}JA_{\xi}J(X)}{Y} - \innpr{JA_{\xi}JA_{\xi}(X)}{Y}.
	\end{align*}
	Thus, $*\Ric$ is symmetric if and only if $A_{\xi}JA_{\xi}J = JA_{\xi}JA_{\xi}$. Applying this operator to a principal vector $v \in E_{\lambda}$, we require
	$$
	A_{\xi}(JA_{\xi}Jv) = JA_{\xi}J(A_{\xi}v) = \lambda(JA_{\xi}Jv).
	$$
	This holds if and only if $JA_{\xi}Jv$ is an eigenvector of $A_\xi$ with eigenvalue $\lambda$, i.e., $JA_{\xi}J(E_{\lambda}) \subseteq E_{\lambda}$.
	
	(2) If $M$ is weakly $*$-Einstein, there exists a function $\rho$ such that $*\Ric(X, Y) = \rho \innpr{X}{Y}$. By \eqref{star ricci}, this implies
	$$
	\innpr{JA_{\xi}JA_{\xi}(X)}{Y} = (1-\rho)\innpr{X}{Y}.
	$$
	Letting $c = 1-\rho$, we have the operator identity $JA_{\xi}JA_{\xi} = c\operatorname{Id}$, or equivalently $A_{\xi}JA_{\xi} = -cJ$. For any $v \in E_{\lambda}$,
	$$
	-cJv = A_{\xi}JA_{\xi}(v) = \lambda A_{\xi}(Jv).
	$$
	If $0$ is a principal curvature, choosing a non-zero eigenvector $v \in E_0$ yields $c J v = 0$, which forces $c=0$. In contrast, if all principal curvatures are non-zero, it follows that $A_{\xi}(Jv) = -\frac{c}{\lambda} Jv$. This implies $J(E_{\lambda}) \subseteq E_{-c/\lambda}$ where the non-zero principal curvature $-\frac{c}{\lambda}$ ensures $c\neq 0$. Furthermore, since $J(E_{-c/\lambda})\subseteq E_{\lambda}$ holds similarly, we conclude that $J(E_{\lambda}) = E_{-c/\lambda}$. 
	
		(3) From part (2), the behavior of the principal curvatures depends on the function $c = 1 - \rho$. If $c\neq 0$ at a point $p$, the eigenvalues of $A_\xi$ come in pairs $(\lambda, -c/\lambda)$ (possibly with $\lambda^2 = -c$) with equal multiplicities $\dim E_{\lambda} = \dim E_{-c/\lambda}$ due to the isomorphism $J: E_{\lambda} \to E_{-c/\lambda}$. In this case, the Gauss--Kronecker curvature $K = \det(A_\xi)$ is computed by grouping these pairs:
	$$
	K = \prod_{\lambda} \lambda^{\dim E_{\lambda}} = \prod_{\text{pairs } \{\lambda, -c/\lambda\}} \left( \lambda \cdot \left(-\frac{c}{\lambda}\right) \right)^{\dim E_{\lambda}} = \prod (-c)^{\dim E_{\lambda}} = (-c)^{n} = (\rho-1)^n.
	$$
	Conversely, if $c=0$ at $p$, part (2) asserts that $0$ is a principal curvature, which implies $K=0=(-c)^{n}$.
	Therefore, $K$ is constant if and only if $c$ (and hence $\rho$) is constant, which is the definition of being $*$-Einstein.

\end{proof}

\bibliographystyle{amsplain}

\end{document}